\documentclass[a4paper]{amsart}
\usepackage{amsfonts, amsmath, amssymb, amscd, amsthm, array, graphicx, geometry, mathtools, xspace, stmaryrd, enumerate}

\usepackage[utf8]{inputenc}
\usepackage[T1]{fontenc}

\usepackage[textsize=footnotesize,textwidth=20ex]{todonotes}

\usepackage[english]{babel}
\usepackage{hyperref}
\usepackage{dirtytalk}
\usepackage{mathrsfs}
\usepackage{geometry}
\newtheorem{thm}{Theorem}[section]
\newtheorem*{thm*}{Theorem}
\newtheorem{cor}[thm]{Corollary}
\newtheorem{lem}[thm]{Lemma}
\newtheorem{prop}[thm]{Proposition}

\theoremstyle{definition}
\newtheorem{dfn}[thm]{Definition}

\newtheorem{obs}[thm]{Observation}

\newtheorem{con}[thm]{Construction}

\newcommand\ZZ{\mathbb{Z}}

\newcommand{\QQ}{{\mathbb Q}}

\newcommand\NN{\mathbb{N}}

\newcommand{\GL}{\operatorname{GL}}

\newcommand{\Aut}{\operatorname{Aut}}

\newcommand{\End}{\operatorname{End}}
\newcommand{\Inn}{\operatorname{Inn}}

\newcommand{\Fix}{\operatorname{Fix}}

\newcommand{\gen}[1]{\left\langle#1\right\rangle}

\usepackage[most]{tcolorbox} 
    \definecolor{block-gray}{gray}{0.97}
    \newtcolorbox{qq}[1][]{%
        colback=block-gray,
        grow to right by=-10mm,
        grow to left by=-10mm, 
        boxrule=0pt,
        boxsep=2pt,
        enhanced jigsaw,
        borderline west={4pt}{0pt}{gray},
        #1,
    }
    \newtcolorbox{qqo}[1][]{%
        colback=block-gray,
        boxrule=0pt,
        boxsep=2pt,
        enhanced jigsaw,
        borderline west={4pt}{0pt}{gray},
        #1,
    }

 \newcommand{\pres}[2]{\left\langle#1 \mid #2\right\rangle}
\newcommand{\ncl}[1]{\ll \!#1 \!\gg}

\newcounter{mallikacomments}

\begin{document}

\title{Endo-Twisted Conjugacy and Outer Fixed Points in Solvable Baumslag--Solitar Groups}


\author{Mallika Roy}
\address{Harish-Chandra Research Institute, HBNI, Chhatnag Road, Jhunsi, Prayagraj (Allahabad) 211019, India.} \email{mallikaroy@hri.res.in; mallikaroy75@gmail.com}



\begin{abstract}
In this article, we solve the twisted conjugacy problem with respect to endomorphisms for solvable Baumslag--Solitar groups $BS(1,n)$, i.e., we propose an algorithm which, given two elements $u,v \in BS(1,n)$ and an endomorphism $\psi \in \End(BS(1,n))$, decides whether $v=(x\psi)^{-1} u x$ for some $x\in BS(1,n)$. Also, we connect the outer fixed points of a given endomorphism $\psi$ with $\varphi$-twisted conjugacy problem for two words $u, v \in BS(1,n)$, where $\varphi \in \End(BS(1,n))$ and $u, v$ depend on $\psi$. Furthermore, we define the weakly (outer) fixed points and discuss its interplay with the endo-twisted conjugacy problem in $BS(1, n)$.
\end{abstract}

\maketitle

\section{Introduction}


Let $G$ be a group, and $\varphi\in \Aut(G)$ be an automorphism of $G$. Two elements $u, v \in G$ are said to be $\varphi$-twisted conjugated, if there exists $x \in G$ such that $u = (x\varphi)^{-1} u x $, and we denote it by $u \sim_\varphi v$. This twisted conjugacy was first introduced by Reidemeister~\cite{R} from topological view point, and plays a pivotal role in modern Nielsen fixed point theory. One can observe that $\sim_\varphi$ is an equivalence relation on $G$, and it coincides with the usual conjugation, when we replace $\varphi$ by the identity homomorphism.

In a similar manner, one can define endo-twisted conjugacy --- for a given endomorphism $\psi \in \End(G)$ and two elements $u, v \in G$, we say that $u, v$ are $\psi$-twisted conjugated, denoted by $u \sim_\psi v$, if there exists $x \in G$ such that $u = (x\psi)^{-1} u x $. The algorithmic approach of endo-twisted conjugacy is the following:

\noindent \textbf{\boldmath Endo-Twisted Conjugacy Problem, $E-TCP(G)$:} For a finite presentation $G=\pres{X}{R}$, \emph{given an endomorphism $\psi\colon G\to G$ (given by the images of the generators), and two words on the generators, $u,v\in F(X)$, decide whether $u$ and $v$ represent $\psi$-twisted conjugate elements in $G$, i.e., $u\sim_{\psi} v$.}

The central topic of this article is the endo-twisted conjugacy problem for solvable Baumslag--Solitar groups, $BS(1,n)$.

Instead of an endomorphism, if we execute the algorithm with an automorphism as a part of input, then it is known as auto-twisted conjugacy problem or simply by twisted conjugacy problem of $G$, in short TCP(G). The study of Algorithmic Group Theory was initiated with three Dehn Problems (see~\cite{Dehn}), back in 1911, namely the word problem, the conjugacy problem and the isomorphism problem. $E-TCP(G)$ or $TCP(G)$ generalizes the so-called conjugacy problem, $CP(G)$ as discussed above. A positive solution to the twisted conjugacy problem instinctively provides a solution to the conjugacy problem, and hence a solution to the word problem. On the contrary, there exist finitely presented groups with solvable word problem but unsolvable conjugacy problem (see~\cite{M}). Similarly, there exists a finitely presented group with solvable conjugacy problem, but unsolvable twisted conjugacy problem (see~\cite[Corollary 4.9]{BMV}).

It is well-known that the celebrated three Dehn problems are unsolvable in their full generality, but there are substantial intriguing results do exist in the literature solving these three problems in certain families of groups. Let, $1 \rightarrow F \rightarrow G \rightarrow H \rightarrow 1$ be a given short exact sequence of groups with certain conditions, specifically $F$ has solvable auto-twisted conjugacy problem. Bogoploski--Martino--Ventura~\cite{BMV} proved that $G$ has solvable conjugacy problem if and only if the corresponding action subgroup $A \leqslant \Aut(F)$ is orbit decidable. As applications, it was shown that the auto-twisted conjugacy problem is solvable for virtually surface groups and for virtually polycyclic groups in~\cite{BMV}.
There has been an extensive study based on the auto-twisted conjugacy problem. The auto-twisted conjugacy problem for free groups is decidable~\cite{BMMV} and also studied in the other families of groups~\cite{BV}, \cite{BuMV}, \cite{Cr}, \cite{Cr2}, \cite{FTV}, \cite{GM-V}, \cite{Macedo-Yuri}, \cite{Roy-Ventura-Mitra}. 
But there is only sparse literature on the endo-twisted conjugacy problem. The
endo-twisted conjugacy problem for free groups is known to be decidable for certain non-surjective maps~\cite{Kim16}, in general it is an open problem.

In this article, we study the endo-twisted conjugated problem for $BS(1,n)$ and later connect $E-TCP$ with the outer fixed points of the endomorphisms of $BS(1,n)$. This connection provides another approach to decide whether two given elements $u,v$ of $BS(1,n)$ are endo-twisted conjugated, i.e., $u \sim_\varphi v$, where $u, v$ and $\varphi \in \End(BS(1,n))$ are determined by a suitably constructed endomorphism $\psi \in \End(BS(1,n))$. The later part of this work was motivated by~\cite{Laura-Logan}, where Ciobanu--Logan computed the outer fixed points (not being a proper power) in order to find the primitive element which generates the non-trivial fixed subgroup of monomorphisms of free groups of rank $2$.


The article is organised as follows. We briefly review some preliminaries regarding Baumslag--Solitar groups, specifically solvable Baumslag--Solitar groups, $BS(1,n) = \langle a, t \mid  a  = t^{-1} a^n t\rangle $ in Section~\ref{sec: BS overview}. This brief survey mainly includes a normal form of its elements, and its semi-direct product structure based on the notations developed in~\cite{Roy-Ventura-Mitra}. Afterwards in Subsection~\ref{sec: endos}, we discuss the construction of two types of endomorphisms, following the results from Collins--Levin~\cite{CoLe} and O'Neil~\cite{N}.
In Section~\ref{sec: TCP}, we prove that the endo-twisted conjugacy problem $E-TCP$ (see Theorem~\ref{thm: main TCP}) is decidable in $BS(1,n)$. In this aspect, we would like to mention that in~\cite{Roy-Ventura-Mitra} Mitra--Roy--Ventura solved the auto-twisted conjugacy problem in $BS(1,n)$.
Finally, in Section~\ref{sec: outer fixed points} we construct an endomorphism $\psi$ in a way so that the intersection $[a] \cap \Fix \psi$ is non-empty if and only if certain $u$ and $v$ are $\varphi$-twisted conjugated, where $u$ and $v$, $\varphi \in \End(BS(1,n))$ are determined by that $\psi$ (see Proposition~\ref{thm: TCP FP}). Later, we define the weakly outer fixed points and prove (see Proposition~\ref{prop: TCP WFP}) that the existence of weakly outer fixed points of an endomorphism $\psi$ assures the twisted conjugacy of two elements defined by $\psi$.

\subsection*{General notation and conventions}

For a group $G$, $\End(G)$ (resp., $\Aut(G)$) denotes the group of endomorphisms (resp., automorphisms) of $G$. We write them all with the argument on the left, that is, we denote by $(x)\varphi$ (or simply $x\varphi$) the image of the element $x$ by the homomorphism $\varphi$; accordingly, we denote by $\varphi \psi$ the composition $A \xrightarrow{\varphi} B \xrightarrow{\psi} C$. Specifically, we will reserve the letter $\gamma$ for right conjugations, $\gamma_g \colon G \to G, x\mapsto g^{-1}xg$. Following the same convention, when thinking of a matrix $A$ as a map, it will always act on the right of horizontal vectors, $\textbf{v}\mapsto \textbf{v}A$. We denote by $\GL_m(\ZZ)$ the linear group over the integers. We use the function $| \cdot |_a$ to count the number of $a$-occurrences in a word $w$.

\section{Preliminaries on Solvable Baumslag--Solitar groups}\label{sec: BS overview}

The so-called Baumslag--Solitar groups $BS(m,n)$ are the class of two-generated one-related groups presented by 
\[
BS(m,n) = \pres{a,t}{a^m  = t^{-1} a^n t},
\]
for $m,n\in \ZZ$. They all are HNN-extensions of $\ZZ$ (with associated subgroups $m\ZZ$ and $n\ZZ$) sitting in the middle of the well known short exact sequence
 \begin{equation}\label{eq: ses}
\begin{array}{ccccccccc} 1 & \to\, & \ncl{a} & \, \to & BS(m,n) & \stackrel{\pi}{\to} & \ZZ =\gen{t} & \to & 1. \\ & & & & a & \mapsto & 1 & & \\ & & & & t & \mapsto & t & &
\end{array}
 \end{equation}
The homomorphism $\pi\colon BS(m,n)\twoheadrightarrow \ZZ$, $w\mapsto t^{|w|_t}$, is well defined because the defining relation is $t$-balanced (not being the case for $a$, unless $m=n$). We note that $\ker \pi =\ncl{a}\, =\{w(a,t) \mid |w|_t =0\}\unlhd BS(m,n)$. The elementary examples include $BS(0,1)\simeq \ZZ$, $BS(1,1)\simeq \ZZ^2$, and $BS(1,-1)=\pres{a,t}{a=t^{-1}a^{-1}t}$, which is the fundamental group of the Klein bottle.

From now on, we concentrate in the solvable groups within this family, and that is the case when $m=1$. Also, for the rest of the paper, we assume $n\neq \pm 1$. 

\subsection{Normal Form}

In this Section we recall the normal form of an element in $BS(1, n)$ developed in~\cite[Section 2.1]{Roy-Ventura-Mitra} and we fix some notations which will be used throughout the paper. The defining relation of $BS(1,n)$ is $a = t^{-1} a^n t$, which can also be thought as of the forms $ta = a^n t$ and $ta^{-1} = a^{-n}t$ --- allows us to accumulate each positive occurrence of $t$ to right. Repetitively using these relations, one can easily obtain:
\begin{equation}\label{eq: jump right}
t^p a^k =a^{kn^p}t^p,
 \end{equation}
for every $p,k\in \ZZ$, $p\geq 0$. Similarly, rewriting the defining relation as $at^{-1} = t^{-1}a^n $ and $a^{-1}t^{-1} = t^{-1}a^{-n}$ tells us that each negative occurrence of $t$'s can always be moved to the left, and so
\begin{equation}\label{eq: jump left}
a^k t^{-p}=t^{-p}a^{kn^p},
 \end{equation}
for every $p,k\in \ZZ$, $p\geq 0$.

Hence, any element $w(a,t)\in BS(1,n)$ can be expressed in the form $g=t^{-p}a^k t^q$, where $p,k,q\in \ZZ$, $p,q\geq 0$. Since, $BS(1,n)$ is torsion-free, we note that in the group $BS(1,n)$, $t^{-p_1}a^{k_1} t^{p_1+c_1}=t^{-p_2}a^{k_2} t^{p_2+c_2}$ if and only if $c_1 =c_2 \in \ZZ$ and $k_1/n^{p_1}=k_2/n^{p_2}\in \ZZ[1/n]$, where $p_i, k_i, c_i\in \ZZ$ with $p_i\geq 0$, for $i=1,2$. It is clear that $c_1=c_2$ is a necessary condition by applying $\pi$ from~\eqref{eq: ses}.
Thus, any element of $BS(1,n)$ can be written in the form $g= t^{-p}a^k t^q =t^{-p}a^k t^p t^c$, where $p,k, c\in \ZZ$, $p\geq 0$ --- close to being a normal form. 

From now on, we set the following notation: for $\alpha =k/n^p\in \ZZ[1/n]$ (here, $k,p\in \ZZ$, $p\geq 0$) write
 \[a^{\alpha}:=t^{-p}a^k t^p \in BS(1,n).\]
The key ideas for setting up this notation are the following: (1) the equality among rational numbers $kn/n^{p+1} =k/n^p$ corresponds to the equality $t^{-p-1}a^{kn} t^{p+1}=t^{-p} (t^{-1}a^n t)^k t^p=t^{-p}a^k t^p$ among the corresponding elements in $BS(1,n)$; (2) this notation unifies the above two equations~\eqref{eq: jump right} and~\eqref{eq: jump left} --- $t^c$, for $c\in \ZZ$ can be shifted to the right of rational powers of $a$ at the price of multiplying its exponent by $n^c$.

We record Lemma~\ref{lem: linearity} and Lemma~\ref{lem: calculs} from~\cite{Roy-Ventura-Mitra} without giving the proofs. Lemma~\ref{lem: linearity} proves linearity for this family of groups: 

\begin{lem}\label{lem: linearity}
The map
 \begin{equation}\label{eq: BS into GL2}
\begin{array}{rcl} \varphi\colon BS(1,n) & \to & \GL_2(\ZZ[1/n])\leqslant \GL_2(\QQ) \\[3pt] a & \mapsto & A=\left(\begin{smallmatrix} 1 & 1 \\ 0 & 1 \end{smallmatrix}\right) \\[3pt] t & \mapsto & T=\left(\begin{smallmatrix} n & 0 \\ 0 & 1 \end{smallmatrix}\right) \end{array}
 \end{equation}
defines a monomorphism of groups. In particular, $BS(1,n)$ is a linear group. 
\end{lem}

Lemma~\ref{lem: calculs} states that the new exponential notation is compatible with the standard rules of computation: 

\begin{lem}\label{lem: calculs}
For any $\alpha =k/n^p,\, \beta=\ell/n^q\in \ZZ[1/n]$ ($k,\ell, p,q\in \ZZ$, $p,q\geq 0$) and $c, r\in \ZZ$, we have 
\begin{itemize}
\item[(i)] $t^c a^{\alpha} =a^{\alpha n^c} t^c$ (equivalently, $a^{\alpha} t^c= t^c a^{\alpha n^{-c}}$);
\item[(ii)] $a^{\alpha}a^{\beta}=a^{\alpha+\beta}$;
\item[(iii)] $(a^{\alpha})^r =a^{r\alpha}$;
\item[(iv)] $t^{-c}a^{\alpha}t^c=a^{\alpha/n^c}$;
\item[(v)] $\big( a^{\alpha}t^c \big)^{-1}=a^{-\alpha n^{-c}} t^{-c}$; 
\item[(vi)] $\big( a^{\alpha}t^c \big)^r = a^{\frac{n^{rc}-1}{n^c-1}\alpha} t^{rc}$.
\end{itemize}
\end{lem}

The reader is referred to~\cite{Roy-Ventura-Mitra} for explicit proofs of both the aforementioned Lemmas.

Using this notation and the above arguments, every element $g\in BS(1,n)$ can be written, in a unique way, as $g=a^{\alpha}t^c$, for some $\alpha\in \ZZ[1/n]$ and some $c\in \ZZ$. Moreover, $\big( a^{\alpha_1}t^{c_1} \big)\big( a^{\alpha_2}t^{c_2} \big) =a^{\alpha_1 +n^{c_1}\alpha_2} t^{c_1+c_2}$.

Finally, we conclude this section by discussing the semi-direct product structure of $BS(1,n)$. Clearly,
\[
\ker \pi =\ncl{a} =\{a^{\alpha}t^c\in BS(1,n) \mid c=0\}=\{a^{\alpha} \mid \alpha\in \ZZ[1/n]\}\simeq \ZZ[1/n].
\]
And so, the short exact sequence~\eqref{eq: ses} for $m=1$ has the form

\begin{equation*}
\begin{array}{ccccccccc} 1 & \to\, & \ZZ[1/n] & \, \to & BS(1,n) & \to & \ZZ & \to & 1,
\end{array}
 \end{equation*}
and it splits because $\ZZ$ is free. In particular, $BS(1,n) \simeq \ZZ[1/n] \rtimes \ZZ$, where $\ZZ = \gen{t}$ acts on $\ZZ[1/n]$ by multiplying by $n$. We also observe that $BS(1,n)^{\rm ab}=\ZZ \oplus \ZZ/|n-1|\ZZ$, with $\ncl{a}$ being, for $n\neq 1$, the only normal subgroup whose quotient is isomorphic to $\ZZ$.

\subsection{\boldmath Endomorphisms of $BS(1,n)$}\label{sec: endos}

Following a previous work by Collins--Levin~\cite{CoLe}, J. O'Neill~\cite{N} explicitly described two types of endomorphisms, namely \textit{type-I} and \textit{type-II}, of $BS(1,n)$, where the first type, \textit{type-I} are monomorphisms and the second type, \textit{type-II} contains the subclass of proper retracts. We state here the result without giving its proof. 

\begin{prop}[J. O'Neill, {\cite[Proposition 2.1]{N}}]\label{prop: O'Neil}
For $n\neq \pm 1$,
 $$
\End(BS(1,n))=\{\psi^{I}_{\alpha, \beta, 1},\, \psi^{II}_{0, \beta, c} \mid \alpha, \beta\in \ZZ[1/n],\, c \in \ZZ\}
 $$
where, for $\alpha=\ell/n^k$, $\beta=q/n^p$ and $c = r -p$ with $k,\ell, p,q, r\in \ZZ$, $k, p, r\geq 0$, $\psi^{I}_{\alpha, \beta, 1}$ and $\psi^{II}_{0, \beta, c}$ are defined as follows:
\begin{equation}\label{eq: endos I}
\begin{array}{rcl} \psi^{I}_{\alpha, \beta, 1}\colon BS(1,n) & \to & BS(1,n) \\ a & \mapsto & a^{\alpha}=t^{-k} a^{\ell} t^k \\ t & \mapsto & a^{\beta}t=t^{-p}a^{q} t^{p+1};
 \end{array}
 \end{equation}
\begin{equation}\label{eq: endos II}
\begin{array}{rcl} \psi^{II}_{0, \beta, c}\colon BS(1,n) & \to & BS(1,n) \\ a & \mapsto & a^{0}=1 \\ t & \mapsto & a^{\beta}t^c=t^{-p}a^{q} t^{r} = t^{-p}a^{q}t^pt^{r-p},
 \end{array}
 \end{equation}
for  $\psi^{II}_{0, \beta, c}$, if $q$ is a multiple of $n$, then $p=0$ or $r=0$.
\end{prop}

\begin{obs}
   Let $\psi^{II}_{0, \beta, c}$ be a type-II endomorphism, where $\beta = q/n^p$ and $c = r -p$ with $p,q, r\in \ZZ$, $p, r\geq 0$. If $q$ is a multiple of $n$, i.e., $q = kn$, for some $k \in \NN$, then $p=0$ or $r=0$. Further, if $p=0$, then $\psi^{II}_{0, \beta, c} = \psi^{II}_{0, \beta', c'}$, where $\beta' = kn$ and $c' =r$ --- $\psi^{II}_{0, \beta', c'} \, \colon a \mapsto 1, \, t \mapsto a^{kn}t^r$. And if $r=0$, then $\psi^{II}_{0, \beta, c} = \psi^{II}_{0, \beta'', c''}$, where $\beta'' = q/n^p$ and $c'' =-p$ --- $\psi^{II}_{0, \beta'', c''} \, \colon a \mapsto 1, \, t \mapsto a^{q/n^p} t^{-p}$. Thus, independent of the fact whether $q$ is a multiple of $n$ or not, any type-II endomorphism of $BS(1,n)$ is of the form $\psi^{II}_{0, \beta, c}$, where $\beta \in \ZZ[1/n]$ and $c \in \ZZ$.
\end{obs}

The following Proposition captures how any two endomorphisms in $BS(1,n)$ compose in the our notation of `normal form'.

\begin{lem}\label{lem: Composition}
The composition rules of type-I and type-II endomorphisms are the following:
\begin{itemize}
    \item[(i)] $\psi^{I}_{\alpha_1, \beta_1, 1} \circ \psi^{I}_{\alpha_2, \beta_2, 1} =\psi^{I}_{\alpha_1\alpha_2, \beta_1\alpha_2+\beta_2, 1}$;
    
    \item[(ii)] $\psi^{II}_{0,\beta_1, c_1} \circ \psi^{II}_{0,\beta_2, c_2} =\psi^{II}_{0, \mu\beta_2, c_1c_2}$, where $\mu = \frac{n^{c_1c_2}-1}{n^{c_2}-1}$;
    
    \item[(iii)] $\psi^{II}_{0,\beta_1, c_1} \circ \psi^{I}_{\alpha_2, \beta_2, 1} = \psi^{II}_{0, \beta_1\alpha_2 + \mu_1\beta_2, c_1}$, where $\mu_1 = \frac{n^{c_1} - 1}{n-1}$;
    
    \item[(iv)] $\psi^{I}_{\alpha_1, \beta_1, 1} \circ \psi^{II}_{0,\beta_2, c_2} = \psi^{II}_{0,\beta_2, c_2}$.
\end{itemize}
\end{lem}

\begin{proof}
First we check that both the maps type-I, $\psi^{I}_{\alpha, \beta, 1}$ and type-II, $\psi^{II}_{0, \beta, c}$ are well defined by proving that they preserve the defining relation:
 \begin{align*}
a \stackrel{\psi^I_{\alpha,\beta, 1}}{\quad \mapsto\quad} & a^{\alpha}, \\ 
t^{-1}a^nt \stackrel{\psi^I_{\alpha,\beta, 1}}{\quad \mapsto\quad} &\big( a^{\beta} t\big)^{-1} \big( a^{\alpha}\big)^n \big( a^{\beta} t\big)= t^{-1} a^{-\beta} a^{n\alpha} a^{\beta}t= t^{-1} a^{n\alpha} t =a^{\alpha}.
 \end{align*}
It is very straightforward for the type-II:
 \begin{align*}
a \stackrel{\psi^{II}_{0,\beta, c}}{\quad \mapsto\quad} & 1, \\ 
t^{-1}a^nt \stackrel{\psi^{II}_{0,\beta, c}}{\quad \mapsto\quad} &\big( a^{\beta} t^c\big)^{-1}\, 1\, \big( a^{\beta} t^c\big)= 1.
 \end{align*}

Now let us check that the composition $\psi^I_{\alpha_1, \beta_1, 1} \circ \psi^I_{\alpha_2, \beta_2, 1}$ equals $\psi^I_{\alpha_1\alpha_2, \beta_1\alpha_2+\beta_2, 1}$. In fact, for $\alpha_1={\ell}_1/n^{k_1}$, $\beta_1=q_1/n^{p_1}$ and $\alpha_2={\ell}_2/n^{k_2}$, $\beta_2=q_2/n^{p_2}$, with $k_i, \ell_i, p_i, q_i\in \ZZ$, $k_i, p_i\geq 0$, $i=1,2$, we have
 \begin{align*}
a \stackrel{\psi^I_{\alpha_1, \beta_1, 1}}{\qquad\mapsto\qquad} t^{-k_1} a^{\ell_1} t^{k_1} \stackrel{\psi^I_{\alpha_2, \beta_2, 1}}{\qquad\mapsto\qquad} & \big( a^{\beta_2} t\big)^{-k_1} \big( a^{\alpha_2}\big)^{\ell_1} \big( a^{\beta_2} t\big)^{k_1}= \\ &= a^{\frac{n^{-k_1}-1}{n-1}\beta_2} t^{-k_1} a^{\ell_1\alpha_2} a^{\frac{n^{k_1}-1}{n-1}\beta_2} t^{k_1} \\ &= a^{\frac{n^{-k_1}-1}{n-1}\beta_2} a^{n^{-k_1} \ell_1\alpha_2} a^{n^{-k_1}\frac{n^{k_1}-1}{n-1}\beta_2} \\ &= a^{\frac{n^{-k_1}-1}{n-1}\beta_2+\alpha_1\alpha_2+\frac{1-n^{-k_1}}{n-1}\beta_2} \\ &= a^{\alpha_1\alpha_2},
 \end{align*}
and
\begin{align*}
t \stackrel{\psi^I_{\alpha_1, \beta_1, 1}}{\qquad\mapsto\qquad} t^{-p_1} a^{q_1} t^{p_1+1} \stackrel{\psi^I_{\alpha_2, \beta_2, 1}}{\qquad\mapsto\qquad} & \big( a^{\beta_2} t\big)^{-p_1} \big( a^{\alpha_2} \big)^{q_1} \big( a^{\beta_2} t\big)^{p_1+1}= \\ &= a^{\frac{n^{-p_1}-1}{n-1}\beta_2} t^{-p_1} a^{q_1\alpha_2} a^{\frac{n^{p_1+1}-1}{n-1}\beta_2} t^{p_1+1} \\ &= a^{\frac{n^{-p_1}-1}{n-1}\beta_2} a^{n^{-p_1}q_1\alpha_2} a^{n^{-p_1}\frac{n^{p_1+1}-1}{n-1}\beta_2} t \\ &= a^{\frac{n^{-p_1}-1}{n-1}\beta_2+\beta_1\alpha_2+\frac{n-n^{-p_1}}{n-1}\beta_2} t 
\\ &= a^{\beta_1\alpha_2+\beta_2} t.
 \end{align*}
Therefore, $\psi^I_{\alpha_1,\beta_1, 1} \circ \psi^I_{\alpha_2,\beta_2, 1} =\psi^I_{\alpha_1\alpha_2, \beta_1\alpha_2+\beta_2, 1}$, as we wanted to see in (i).

The following part of the proof contains the composition rule when both the endomorphisms are of type-II:
 \begin{align*}
a \stackrel{\psi^{II}_{0, \beta_1, c_1}}{\qquad\mapsto\qquad} 1 \stackrel{\psi^{II}_{0, \beta_2, c_2}}{\qquad\mapsto\qquad} & 1 ,
 \end{align*}
and
\begin{align*}
t \stackrel{\psi^{II}_{0, \beta_1, c_1}}{\qquad\mapsto\qquad}  a^{\beta_1} t^{c_1} \stackrel{\psi^{II}_{0, \beta_2, c_2}}{\qquad\mapsto\qquad} & \big( a^{\beta_2}t^{c_2}\big)^{c_1} \, \, = \, \,  a^{\frac{n^{c_1c_2}-1}{n^{c_2}-1}\beta_2} t^{c_1 c_2} \, \, = \, \, a^{\mu\beta_2} t^{c_1 c_2}.
 \end{align*}
Hence, $\psi^{II}_{0,\beta_1, c_1} \circ \psi^{II}_{0,\beta_2, c_2} =\psi^{II}_{0, \mu\beta_2, c_1c_2}$, as announced in (ii).

To prove (iii):
 \begin{align*}
a \stackrel{\psi^{II}_{0, \beta_1, c_1}}{\qquad\mapsto\qquad} 1 \stackrel{\psi^{I}_{\alpha_2, \beta_2, 1}}{\qquad\mapsto\qquad} & 1 ,
 \end{align*}
and
\begin{align*}
t \stackrel{\psi^{II}_{0, \beta_1, c_1}}{\qquad\mapsto\qquad}  t^{-p_1}a^{q_1}t^{p_1}t^{c_1} \stackrel{\psi^{I}_{\alpha_2, \beta_2, 1}}{\qquad\mapsto\qquad} & = \big( a^{\beta_2} t\big)^{-p_1} \big( a^{\alpha_2} \big)^{q_1} \big( a^{\beta_2} t\big)^{p_1} \big( a^{\beta_2} t\big)^{c_1}\\ &=  a^{\frac{n^{-p_1}-1}{n-1}\beta_2} t^{-p_1} a^{q_1\alpha_2} a^{\frac{n^{p_1}-1}{n-1}\beta_2} t^{p_1} a^{\frac{n^{c_1}-1}{n-1}\beta_2} t^{c_1} \\&= a^{\frac{n^{-p_1}-1}{n-1}\beta_2} a^{q_1 \alpha_2 n^{-p_1}} a^{n^{-p_1}} a^{\frac{n^{p_1}-1}{n-1}\beta_2} t^{p_1} a^{\frac{n^{c_1}-1}{n-1}\beta_2} t^{c_1}\\
&= a^{\frac{n^{-p_1}-1}{n-1}\beta_2} a^{\beta_1\alpha_2} a^{\frac{1 - n^{-p_1}}{n-1}\beta_2} a^{\frac{n^{c_1}-1}{n-1}\beta_2} t^{c_1} \\ &= a^{\beta_1\alpha_2+\mu_1\beta_2} t^{c_1}.
 \end{align*}
This completes the proof of (iii).

Finally,
 \begin{align*}
a \stackrel{\psi^I_{\alpha_1, \beta_1, 1}}{\qquad\mapsto\qquad} t^{-k_1} a^{\ell_1} t^{k_1} \stackrel{\psi^{II}_{0, \beta_2, c_2}}{\qquad\mapsto\qquad} & \big( a^{\beta_2} t^{c_2}\big)^{-k_1} 1 \big( a^{\beta_2} t^{c_2}\big)^{k_1}= 1,
 \end{align*}
and
\begin{align*}
t \stackrel{\psi^I_{\alpha_1, \beta_1, 1}}{\qquad\mapsto\qquad} t^{-p_1} a^{q_1} t^{p_1}t \stackrel{\psi^{II}_{0, \beta_2, c_2}}{\qquad\mapsto\qquad} & \big( a^{\beta_2} t^{c_2}\big)^{-p_1} 1 \big( a^{\beta_2} t\big)^{p_1} \big( a^{\beta_2} t^{c_2}\big) \, \, = \, \, \big( a^{\beta_2} t^{c_2}\big).
 \end{align*}
This completes the proof of (iv).
\end{proof}

Using Lemma~\ref{lem: calculs}, below we compute the images of an arbitrary element $a^{\nu}t^d\in BS(1,n)$ (with $\nu=m/n^r\in \ZZ[1/n]$, $m,r,d\in \ZZ$, $r\geq 0$) by type-I and type-II endomorphisms respectively.

 \begin{align}\label{eq: action I}
\big( a^{\nu}t^d \big) \psi^I_{\alpha,\beta, 1} = \big( t^{-r}a^{m}t^{r+d}\big) \psi^I_{\alpha,\beta, 1} &= \big( a^{\beta}t\big)^{-r} a^{m\alpha} \big(a^{\beta}t\big)^{r+d} \notag \\ &= a^{\frac{n^{-r}-1}{n-1}\beta}t^{-r} a^{m\alpha} a^{\frac{n^{r+d}-1}{n-1}\beta}t^{r+d} \notag \\ &= a^{\frac{n^{-r}-1}{n-1}\beta} a^{m n^{-r}\alpha} a^{n^{-r}\frac{n^{r+d}-1}{n-1}\beta}t^{d} \\ &= a^{\frac{n^{-r}-1}{n-1}\beta+ m n^{-r}\alpha + \frac{n^{d}-n^{-r}}{n-1}\beta}t^{d} \notag \\ &= a^{\nu \alpha + \frac{n^{d}-1}{n-1}\beta}t^{d}. \notag
 \end{align}

  \begin{align}\label{eq: action II}
\big( a^{\nu}t^d \big) \psi^{II}_{0,\beta, c} = \big( t^{-r}a^{m}t^{r+d}\big) \psi^{II}_{0,\beta, c} &= \big( a^{\beta}t^c\big)^{-r} 1 \big( a^{\beta}t^c\big)^{r+d} \notag \\ &= \big( a^{\beta}t^c\big)^d \\ &= a^{\frac{n^{cd}-1}{n^c-1}\beta}t^{cd}. \notag
 \end{align}

\section{The endo-twisted conjugacy problem in $BS(1,n)$}\label{sec: TCP}

In this section, our goal is to prove that the endo-twisted conjugacy problem is solvable in $BS(1,n)$. First, we record the following elementary fact, which holds in an arbitrary group $G$. 

\begin{lem} \label{lem: twisted endos}
Let $G$ be a group, $u,v,w \in G$, $\varphi \in \Aut(G)$, $\psi \in \End(G)$ and $\gamma_w \in \Inn(G)$. Then, the following are equivalent:
 \begin{itemize}
\item[(a)] $u \sim_\psi v$,
\item[(b)] $u \varphi \sim_{\varphi^{-1}\psi\varphi} v \varphi$,
\item[(c)] $wu \sim_{\psi\gamma_{w^{-1}}} wv$, 
\item[(d)] $uw \sim_{\gamma_{w^{-1}}\psi} vw$. \hfill \qed
 \end{itemize}
\end{lem}

\begin{proof}
    $u \sim_\psi v$ if and only if there exists $g \in G$ such that $v = (g\psi)^{-1} u g \Leftrightarrow v \varphi = \big( (g\psi)^{-1} u g\big) \varphi$. Let $g \varphi = g'$. Then, $(g\psi)^{-1} u g\big) \varphi = (g^{-1})\psi\varphi \, u\varphi \, g\varphi = (g^{-1}\varphi)\varphi^{-1}\psi\varphi \, u\varphi \, g\varphi = (g')^{-1}\varphi^{-1}\psi\varphi \, (u\varphi) \, g' = (g' \varphi^{-1}\psi\varphi)^{-1} \, (u\varphi) \, g'$. Thus, we get $v\varphi = (g' \varphi^{-1}\psi\varphi)^{-1} \, (u\varphi) \, g' \Leftrightarrow u\varphi \sim_{\varphi^{-1}\psi\varphi} v\varphi$, which completes the proof of (a) $\Leftrightarrow$ (b).

    Again, suppose (a) holds, and so there exists $g \in G$ such that $v = (g\psi)^{-1} u g$
\begin{align*}
v = (g\psi)^{-1} u g \, \Leftrightarrow \, wv &= w(g\psi)^{-1} \, w^{-1} w \, u \, g \\ & = (g^{-1}\psi) \gamma_{w^{-1}} \, wu \, g \\ & = (g \psi \gamma_{w^{-1}})^{-1} \, wu \, g.
 \end{align*}
Hence, $u \sim_\psi v$ if and only if $wu \sim_{\psi \gamma_{w^{-1}}} wv$ and completes the proof of $(a) \Leftrightarrow (c)$.

Now, suppose (d) holds, i.e., there exists $g \in G$ such that $vw = (g\gamma_{w^{-1}} \psi)^{-1} \, uw \, g$
\begin{align*}
vw & = (g\gamma_{w^{-1}} \psi)^{-1} \, uw\, g\\ & = (g\gamma_{w^{-1}})^{-1} \psi \, uw \, g \\ & = (w g w^{-1})^{-1} \psi \, uw \, g \\&= (w g^{-1} w^{-1})\psi \, uw \,g \\ \Leftrightarrow v & = (w g^{-1} w^{-1})\psi \, u \, (w g w^{-1}) \\ &= (\tilde{g}\psi)^{-1} \, u \, \tilde{g} \\ \Leftrightarrow u &\sim_{\psi} v,
 \end{align*}
where $wgw^{-1} = \tilde{g}$.Thus, we have (d) $\Leftrightarrow$ (a), and we get the announced equivalence result. 
\end{proof}

\begin{thm}\label{thm: main TCP}
The endo-twisted conjugacy problem is solvable in $BS(1,n)$.
\end{thm}

\begin{proof}
    Let us suppose that we are given a fixed endomorphism $\psi \in \End (BS(1,n))$ and two elements $u = a^{\nu_1}t^{d_1}, v = a^{\nu_1}t^{d_1} \in BS(1, n)$, where $\nu_i =m_i/n^{r_i}\in \ZZ[1/n]$ with $m_i, r_i, d_i\in \ZZ$, $r_i\geq 0$, $i=1,2$. We have to decide whether $a^{\nu_1}t^{d_1}$ and $a^{\nu_2}t^{d_2}$ are $\psi$-twisted conjugated to each other, i.e., whether $u\sim_\psi v$. As there are two types of endomorphisms, we need to consider the following two cases, namely, Case-1: $\psi = \psi^{II}_{0, \beta, c}$, and Case-2: $\psi = \psi^{I}_{\alpha, \beta, 1}$.

    \noindent
    \textit{Proof of Case-1:} Let us take $w=t^{d_1}, \text{ and } w'=t^{-r_1}$. $\psi^{II}_{0, \beta, c}\in \End(BS(1,n))$ is given where $\beta = \ell/n^q$, where $\ell, q, c \in \ZZ$ and $q \geq 0$. We write $\psi^{II}_{0, \beta, c}$ simply as by $\psi^{II}$. Applying Lemma~\ref{lem: twisted endos} (a) $\Leftrightarrow$ (d), 
    %
    \begin{equation}\label{equ: change 1}
        u \sim_{\psi} v \Leftrightarrow uw^{-1} \sim_{\gamma_w\psi} vw^{-1} \Leftrightarrow a^{\nu_1} \sim_{\gamma_w\psi} a^{\nu_2} t^{d_2 - d_1} \Leftrightarrow t^{-r_1} a^{m_1} t^{r_1}  \sim_{\gamma_w\psi} t^{-r_2} a^{m_2} t^{r_2}t^{d_2 - d_1}
    \end{equation}
   Now take $\varphi = \gamma_{w'} \in \Aut(BS(1,n))$. Combining Lemma~\ref{lem: twisted endos} (a) $\Leftrightarrow$ (b) and from~\eqref{equ: change 1} we get the following:
   \begin{equation}
       t^{-r_1} a^{m_1} t^{r_1}  \sim_{\gamma_w\psi} t^{-r_2} a^{m_2} t^{r_2}t^{d_2 - d_1} \Leftrightarrow a^{m_1} \sim_{\Psi^{II}} V,
   \end{equation}
where $\Psi^{II} = (\gamma_{w'})^{-1}\gamma_w\psi (\gamma_{w'})$ and $V = (t^{-r_2} a^{m_2} t^{r_2}t^{d_2 - d_1})\varphi = t^{-(r_2 - r_1)}a^{m_2} t^{r_2 - r_1} t^{d_2 - d_1}$. Further, take $r = r_2 - r_1$ and $d = d_2 - d_1$. First, we observe from~\ref{lem: Composition} that any kind of composition of a type-I and a type-II endomorphism, the resulting endomorphism is of type-II.
Thus, adjusting the twisting endomorphism appropriately, we are reduced to checking the endo-twisted conjugation for inputs $a^{m_1}$ and $t^{-r} a^{m_2} t^r t^d$ with $d, m_1, m_2, r \in \ZZ$, $r \geq 0$. Taking the possible $\psi^{II}$-twisted conjugator as $g = a^{\gamma}t^p$, where $\gamma = y/n^x \in \ZZ[1/n]$, $p \in \ZZ$ and using~\eqref{eq: action II} and Lemma~\ref{lem: calculs}(iv) we have the following:
\begin{align*}
    (a^{\gamma}t^p \psi^{II})^{-1} \, a^{m_1} \, a^{\gamma}t^p 
    &= (a^{\frac{n^{pc} -1}{n^c -1}\beta}t^{pc})^{-1} \, a^{m_1} \, a^{\gamma}t^p \\
    &= a^{\frac{n^{-pc} -1}{n^c -1}\beta}t^{-pc} \, a^{m_1} \, a^{\gamma}t^p \\
    &= a^{\frac{n^{-pc} -1}{n^c -1}\beta} \, a^{m_1n^{-pc}} \, a^{\gamma n^{-pc}} \, t^{-pc+p}\\
    &= a^{(\frac{1 - n^{pc}}{n^c -1}\beta + m_1 + \gamma)n^{-pc}} \, t^{-pc+p}.
\end{align*}
Altogether we have $u \sim_{\psi^{II}} v$ if and only if 
%
\begin{equation}\label{main endo II-1}
  a^{m_2/n^r} t^d = a^{(\frac{1 - n^{pc}}{n^c -1}\beta + m_1 + \gamma)n^{-pc}} \, t^{-pc+p}.  
\end{equation}
Hence, equating both sides of~\eqref{main endo II-1}, firstly, we need $d = p-pc$ and also,
\begin{align*}
    \frac{m_2n^{pc}}{n^r} & = \frac{1 - n^{pc}}{n^c -1}\beta + m_1 + \gamma\\
    \Rightarrow m_2n^{pc-r}(n^c - 1) & = (1 - n^{pc})\beta + m_1(n^c - 1) + \frac{y}{n^x}(n^c - 1)\\
    \Rightarrow m_2(n^c - 1)n^{pc+x} & = (1 - n^{pc})\beta n^r n^x + m_1(n^c - 1) n^r n^x + (n^c - 1) n^r y.
\end{align*}
Let us apply the change of variable $z = pc + x\in \ZZ$, where $c$ is given; then the last equation of the above computation becomes
\begin{align*}
m_2(n^c - 1)n^{ z} & = (1 - n^{z})\beta n^r n^x + m_1(n^c - 1) n^r n^x + (n^c - 1) n^r y \\
\Rightarrow m_2(n^c - 1)n^{ z} & = \beta n^r n^x - \beta n^r n^{ z} + m_1(n^c - 1) n^r n^x + (n^c - 1) n^r y\\
\Rightarrow An^x + By + Cn^{ z} & =0.
\end{align*}
From the inputs, namely, $\ell, q, m_1, m_2, r, d, n \in \ZZ$, $q, r \geq 0$, $n \neq \pm 1$, we compute three integers
\[
A = \beta n^r + m_1n^r(n^c-1), \, \, B= n^r(n^c -1), \, \, C = m_2(1-n^c) - \beta n^r.
\]

Summarizing $A,B,C\in \ZZ$; and $u = a^{\nu_1} t^{d_1}$ and $a^{\nu_2} t^{d_2}$ are $\psi^{II}_{0, \beta, c}$-twisted conjugated to each other if and only if the equation 
 \begin{equation}\label{equ: main endo II-2}
An^x +By + Cn^z = 0
 \end{equation}
has an integral solution, for the unknowns $x,y,z\in \ZZ$, with $x\geq 0$. If $B=\pm 1$ this is always the case; so, assume $B\neq \pm 1$.

If $C=0$ it is easy to check: if further $B=0$ then equation~\eqref{equ: main endo II-2} has no solution unless $A=0$; and if $B\neq 0$ compute the finite set $\mathcal{N}=\{1,n,n^2, \ldots \}\subseteq \ZZ/B\ZZ$ (by computing the successive powers of $n$ until obtaining the first repetition modulo $B$) and check whether the set $A\mathcal{N}$ contains $0 \pmod{B}$.

Now assume $C\neq 0$, write $C=C' n^s$ with $s\geq 0$ and $C'\neq 0$ not being multiple of $n$, and note that all possible solutions to equation~\eqref{equ: main endo II-2} have $z \geq -s$. So, with the change of variable $\ZZ \ni z'=z+s$, \eqref{equ: main endo II-2} is equivalent to 
  \begin{equation}\label{equ: main endo II-3}
 An^x +By + C'n^{z'} =0,
  \end{equation}
 for the unknowns $x,y,z'\in \ZZ$, with $x, z'\geq 0$. Let us distinguish two more cases. 

\textbf{Sub-case 1: $B=0$}. Our equation~\eqref{equ: main endo II-3} becomes $An^x=C'' n^{ z'}$ (where, $C'' = - C'$), which has the desired solution if and only if $A/C''$ is a rational number being a (positive or negative) power of $n$. This can be checked by looking at the prime factorizations of $A$, $C'$, and $n$, and so decidable.

\textbf{Sub-case 2: $B\neq 0$}. Now, equation~\eqref{equ: main endo II-3} is equivalent to 
  \begin{equation}\label{equ: main endo II-4}
 An^x \equiv C''n^{ z'} \pmod{B},
 \end{equation}
 for the unknowns $x,z'\in \ZZ$, with $x,z'\geq 0$. Compute the finite set $\mathcal{N}=\{1,n,n^2, \ldots \}\subseteq \ZZ/B\ZZ$. Clearly, equation~\eqref{equ: main endo II-4} admits a solution if and only if the subsets $A\mathcal{N}$ and $C''\mathcal{N}$ from $\ZZ/B\ZZ$ intersect non-trivially, which is  decidable. This completes the proof of Case-1.

\noindent
\textit{Proof of Case-2:} Let $\psi^{I}_{\alpha, \beta, 1} \in \End(BS(1,n))$, where $\alpha = k/n^j$, $\beta= \ell/ n^q$ with $k, j, \ell, q \in \ZZ$, $j, q \geq 0$, be the given endomorphism. We observe that any type-I endomorphism preserve the $|.|_t$. Hence, $d_1=d_2$ (call it just $d$) is an obvious necessary condition for $u \sim_{\psi^{I}_{\alpha, \beta, 1}} v$. 

Applying Lemma~\ref{lem: twisted endos} (a)$\Leftrightarrow$(d) with $w=t^{-d}$ and (a)$\Leftrightarrow$(b) with $\varphi=\gamma_{t^{-r_1}}$, we have that 
 $$
a^{\nu_1}t^d \sim_{\psi^I_{\alpha, \beta, 1}} a^{\nu_2}t^d \quad \Leftrightarrow \quad a^{\nu_1} \sim_{\gamma_{w^{-1}}\psi^I_{\alpha, \beta,1}} a^{\nu_2} \quad \Leftrightarrow \quad a^{m_1} \sim_{\varphi^{-1}\gamma_{w^{-1}}\psi^I_{\alpha, \beta, 1}\varphi} t^{-(r_2-r_1)}a^{m_2}t^{r_2-r_1}.
 $$
Thus, adjusting the twisting endomorphism appropriately, we are reduced to check twisted conjugation for inputs of the form $u=a^{m_1}$, $v=t^{-r}a^{m_2}t^{r}$, with $m_1, m_2, r (=r_2 - r_1)\in \ZZ$, $r\geq 0$.

Similarly, as in Case-1, writing the possible $\psi^I_{\alpha, \beta, 1}$-twisted conjugator as $g=a^{\gamma}t^p$, $\gamma\in \ZZ[1/n]$, $p\in \ZZ$, and using equation~\eqref{eq: action I}, we have 
 \[
g\psi^I_{\alpha,\beta, 1} =a^{\gamma \alpha + \frac{n^{p}-1}{n-1}\beta}t^{p}; 
\] 
now, using Lemma~\ref{lem: calculs}(iv),  
 \[
\big( g\psi^I_{\alpha,\beta, 1} \big)^{-1} ug=t^{-p} a^{-\gamma \alpha - \frac{n^{p}-1}{n-1}\beta} a^{m_1} a^{\gamma}t^p= a^{n^{-p}\big(-\gamma \alpha -\frac{n^p-1}{n-1}\beta +m_1 +\gamma \big)}.
 \]
Hence, $u=a^{m_1}\sim_{\psi^I_{\alpha, \beta, 1}} t^{-r}a^{m_2}t^r =v$ if and only if 
 \[
\frac{m_2}{n^{r}} n^p =-\gamma \alpha -\frac{n^{p}-1}{n-1}\beta +m_1 +\gamma, 
 \]
for some $\gamma \in \ZZ[1/n]$ and $p\in \ZZ$; rearranging and simplifying, we get
 \[
(n-1)m_2 n^p +(n-1)n^{r}\gamma (\alpha -1) +(n^{p}-1)n^{r}\beta =(n-1)n^{r}m_1.
 \]
Write $\gamma =y/n^x$, $x,y\in \ZZ$, $x\geq 0$, and apply the change of variable $\ZZ \ni z=p+x$; our equation becomes equivalent to the integral equation  
 %
\begin{align*}
 (n-1)m_2 n^p +(n-1)n^{r}\frac{y}{n^x} \frac{k-n^j}{n^j} +(n^{p}-1)n^{r}\frac{\ell}{n^q} & =(n-1)n^{r}m_1,\\ 
 (n-1)m_2 n^{p+x+j+q} +(n-1)y(k-n^j)n^{r+q} +(n^{p}-1)\ell n^{r+j+x} &=(n-1)m_1 n^{r+j+q+x},\\
\Big( \ell n^{r+j}+(n-1)m_1 n^{r+j+q}\Big) n^x +\Big( (n-1)(n^j-k)n^{r+q} \Big) y &=\Big( (n-1)m_2 n^{j+q} +\ell n^{r+j}\Big) n^{z}. 
\end{align*}
From the data, namely $k, j, \ell, q, m_1, m_2, r, n\in \ZZ$, $j,q,r\geq 0$, $n\neq \pm 1$, let us compute the three integers 
 \[
A=\ell n^{r+j}+(n-1)m_1 n^{r+j+q},\qquad B=(n-1)(n^j-k)n^{r+q}, \qquad C=(n-1)m_2 n^{j+q}+\ell n^{r+j}.
 \]
Thus, we get  $A,B,C\in \ZZ$; and $u=a^{m_1}$ and $v=t^{-r}a^{m_2}t^r$ are $\psi^I_{\alpha, \beta, 1}$-twisted conjugated to each other if and only if the equation 
 \begin{equation}\label{eq: equation-I 1}
An^x +By=Cn^z
 \end{equation}
has an integral solution, for the unknowns $x,y,z\in \ZZ$, with $x\geq 0$. If $B=\pm 1$ this is always the case; so, assume $B\neq \pm 1$. Then, we execute the same procedure as of the proof of the Case-1 with these new values of $A, B$ and $C$, in particular we study all the sub-cases, e.g., $C=0$ with $B = 0$ or $B \neq 0$, and $C \neq 0$ with $B = 0$ or $B \neq 0$. And, we see that all of these sub-cases are decidable. This completes the proof of Case-2. 
\end{proof}

\section{Outer fixed points and endo-twisted conjugacy}\label{sec: outer fixed points}

In this section we focus on the outer fixed points of certain kind of endomorphisms $\psi$'s of solvable Baumslag--Solitar groups $BS(1, n)$, and the $\varphi$-twisted conjugacy of two elements $g_1$ and $g_2$ of $BS(1, n)$, where $g_1, g_2$ and $\varphi \in \End(BS(1,n))$ are determined by that particular endomorphism $\psi$. 
We would also like to highlight that in this section we are only considering the type-I endomorphisms. Later, we introduce the notion of weakly (outer) fixed points; and exhibit the interaction between the weakly (outer) fixed points and endo-twisted conjugacy problem in $BS(1, n)$.

\begin{dfn}
    Let $G$ be a group and $\varphi \in \End(G)$. An \emph{outer fixed element} of $\varphi$ is an element $g \in G$ such that $g \varphi \sim g$. 

    An \emph{outer fixed point} of an endomorphism $\varphi \in \End(G)$ is the conjugacy class $[g]$ of an outer fixed element $g$ of $\varphi$. 
\end{dfn}

It is easy to observe that if $g$ is an outer fixed element, then $[g]$ satisfies $([g])\varphi \subseteq [g]$.

For a word $z$, let us define the endomorphism $\varphi_z : BS(1, n) \rightarrow BS(1, n)$ as
$a\varphi_z = a, t\varphi_z = z$. The following Proposition~\ref {thm: TCP FP} connects the existence of fixed points (in the conjugacy class of $a$) of the input map $\psi$ to the $\varphi_z$-twisted conjugacy problem for words $u$ and $a^\alpha$, where $u$ is given with $|u|_t = 0$ but the integer $\alpha \in \ZZ[1/n]$ is unknown. Note that, here we put a condition on the input $u$ that $|u|_t = 0$, which is necessary as $|a^\alpha|_t = 0$ and in Section~\ref{sec: TCP} we already have seen that if two words $u$ and $v$ are twisted conjugated by a type-I endomorphism then $|u|_t = |v|_t$.

In this section, we write any type-I endomorphism simply as $\psi^{I}$ without explicitly mentioning its subscripts; and $\varphi_z$ is always a type-I endomorphism of $BS(1,n)$.

\begin{thm}\label {thm: TCP FP}
    Let $u = a ^\nu$ and $v = a^\beta t$ be two given elements of $BS(1,n)$, where $\nu = \frac{m}{n^r} \in \ZZ[1/n], \, \beta = \frac{\ell}{n^q}$, $m, r, \ell, q \in \ZZ$ and $0 \leq r \leq q$. Also, let $\psi^{I} \in \End(BS(1,n)) \colon a \mapsto uau^{-1}, t \mapsto v$. Define $z = u^{-1} v u$.

    There exist $w \in BS(1,n)$ and $\alpha \in \ZZ[1/n]$ such that 
    \begin{equation}\label{eq: TCP OP}
    u^{-1} = (w\varphi_z)^{-1} a^{\alpha} w
    \end{equation}
 if and only if $[a] \cap \Fix \psi^{I}$ is non-empty.  
\end{thm}

\begin{proof}
First we assume that there exists $w \in BS(1,n)$ and $\alpha \in \ZZ[1/n]$ such that~\eqref{eq: TCP OP} holds. As $w$ is an element of $BS(1,n)$, it is a word on $a \text{ and }t$,
$$
\begin{array}{rl}
   (w^{-1} a w)\psi^{I} \, = & w^{-1}(a\psi^{I}, t\psi^{I}) \, a\psi^{I} \,  w(a\psi^{I}, t\psi^{I}) \\
    = & w^{-1}(uau^{-1}, uzu^{-1}) \,  uau^{-1} \,  w(uau^{-1}, uzu^{-1})\\
    = & u \, w^{-1}(a,z) \, u^{-1} \, u a u^{-1} \, u\, w(a, z)\, u^{-1}\\
    = & u (w\varphi_z)^{-1} \, a \, (w\varphi_z) u^{-1}\\
    = & w^{-1} a^{-\alpha} a a^{\alpha} w\\
    = & w^{-1} a w.
\end{array}
$$
Note that the second last equality follows from~\eqref{eq: TCP OP}. Thus, we establish $[a] \cap \Fix \psi^{I} \neq \emptyset$ in the presence of \eqref{eq: TCP OP}.

Let $x \in [a] \cap \Fix \psi^{I}$, also let $w \in BS(1,n)$ such that $x = w^{-1} a w$. Then we have the following
%
\[
w^{-1} a w \, = \, (w^{-1} a w) \psi^{I} \, = \, u(w\varphi_z)^{-1} \, a \, (w\varphi_z)u^{-1},
\]
where the first equality holds as $w^{-1} a w \in \Fix \psi^{I}$, and the second equality comes from the previous calculation. This in turn implies that 
\[
(w\varphi_z)u^{-1}w^{-1} a\,  w\, u(w\varphi_z)^{-1} \, = \, a.
\]
Thus, we get $(w\varphi_z)u^{-1}w^{-1}$ centralises the generator $a$. As $BS(1,n) \simeq \ZZ[1/n] \rtimes \ZZ$, there exist some $\alpha \in \ZZ[1/n]$ such that $(w\varphi_z)u^{-1}w^{-1} = a^\alpha$, and so $u^{-1} = (w\varphi_z)^{-1} a^{\alpha} w$. Hence, the announced result follows.
\end{proof}

\begin{cor}
    Let $\alpha \in \ZZ[1/n]$ be given and $\psi^{I} \in \End(BS(1, n))$ be defined as in~\ref {thm: TCP FP}. Then it is algorithmically decidable whether $[a] \cap \Fix\psi^I$ is non-empty.
\end{cor}

\begin{proof}
    We note that from the construction of $\psi^I$, $u \in BS(1,n)$ and $\varphi_z \in \End(BS(1,n))$ are given; also from the hypothesis $\alpha \in \ZZ[1/n]$ is given. Applying Theorem~\ref{thm: main TCP}, it is algorithmically decidable whether $u^{-1} \sim_{\varphi_z} a^{\alpha}$ for given $u, \alpha, \varphi_z$. Now $u^{-1} \sim_{\varphi_z} a^{\alpha}$ if and only if there exists $w \in BS(1, n)$ such that $u^{-1} = (w\varphi_z)^{-1} a^{\alpha} w$. Then, $w^{-1} a w \in [a] \cap \Fix \psi^{I}$. As we already computed $(w^{-1} a w) \psi^{I} \, = \, u(w\varphi_z)^{-1} \, a \, (w\varphi_z)u^{-1}$, and from Proposition~\ref {thm: TCP FP}, $(w^{-1} a w) \psi^{I} = w^{-1} a w$, i.e., $w^{-1} a w \in [a] \cap \Fix \psi^{I}$, if and only if $u^{-1} \sim_{\varphi_z} a^{\alpha}$ for some $\alpha \in \ZZ[1/n]$. This completes the proof. 
\end{proof}

We note that as $a$ is mapped to $a$ by $\varphi_z$ and $a$ is mapped to $uau^{-1}$ by the newly constructed homomorphism $\psi^{I}$. So, we are only focusing on type-I endomorphisms in this section as we mentioned before.

Now, we introduce the notion of weakly (outer) fixed points to study how the existence of weakly (outer) fixed points guarantees the endo-twisted conjugacy.

\begin{dfn}
   Let $G$ be a group and $\varphi \in \End(G)$. A \emph{weakly fixed point} of $\varphi$ is an element $g \in G$ such that $g \varphi \sim a$, where $a$ is a fixed point of $\varphi$. 
\end{dfn}


\begin{con}\label{con: main con}
Let $\widetilde{\psi}^I$ be an arbitrary type-I endomorphism of $BS(1,n)$ such that $a\widetilde{\psi}^I = uau^{-1}$ and $t\widetilde{\psi}^I = v$. Let $x = u^{-1}vu$ and $\varphi_x \in \End(BS(1,n)) \colon a \mapsto a, t \mapsto x$. Note that the constructions of $\widetilde{\psi}^I$ and $\varphi_x$ (a type-I endomorphism) are same as of $\psi^I$ and $\varphi_z$ respectively (see Proposition~\ref {thm: TCP FP}), except that $|u|_t$ is not necessarily equal to $0$. Also, note that if we consider $v = a^{\beta} t$ for some $\beta \in \ZZ[1/n]$, then $x = a^{\beta'} t$ for some $\beta' \in \ZZ[1/n]$; and so, $t\widetilde{\psi}^I = v = uxu^{-1} = u a^{\beta'} t u^{-1}$. Set,  $\mu = \frac{n^d - 1}{n-1}$.

The image of an arbitrary element $a^{\gamma}t^d$ by $\widetilde{\psi}^I$ is the following:
$$
\begin{array}{rl}
    (a^{\gamma}t^d)\widetilde{\psi}^I\, = & u\, a^{\gamma}\, u^{-1} (uxu^{-1})^d \\
    = & u\, a^{\gamma}\, (a^{\beta'}t)^d \, u^{-1}\\
    = & u\, a^{\gamma}\, a^{\mu\beta'}t^d \, u^{-1}\\
    = & u\, a^{\gamma + \mu\beta'}t^d \, u^{-1}
\end{array}
$$
\end{con}

\begin{obs}\label{obs: WFP}
    Let us consider $\widetilde{\psi}^{I} \in \End(BS(1,n))$ defined as above. Then, all the weakly fixed points of $\widetilde{\psi}^{I}$ are of the form $a^\gamma t^d$ where $a^{\gamma + \mu \beta'} t^d \sim a^{\alpha} t^d$ and $a^{\alpha} t^d$ is a fixed point $\widetilde{\psi}^{I}$.
\end{obs}

\begin{dfn}
    Let $G$ be a group and $\varphi \in \End(G)$. Fix $a \in \Fix \varphi$. An element $x \in G$ is a \emph{weakly outer fixed point related to $a$} of $\varphi$ if there exists $g \in G$ such that $(g^{-1} x g)\varphi = g^{-1} a g$.
    In general, $x$ is a \emph{weakly outer fixed point} of $\varphi$ if there exists $g \in G$ such that $(g^{-1} x g)\varphi = g^{-1} a g$ for some fixed point $a$ of $\varphi$.
\end{dfn}

\begin{obs}\label{obs: WFOP}
    If $x$ is a weakly outer fixed point of an endomorphism $\varphi$ of $G$, then $x$ is a weakly fixed point of $\varphi$.
\end{obs}

The following Proposition connects the weakly fixed points of $\widetilde{\psi}^{I}$ and the twisted conjugacy problem in $BS(1,n)$.

\begin{prop}\label{prop: TCP WFP}
    Let $\widetilde{\psi}^I \in \End(BS(1,n))$ be defined as in~\ref{con: main con} and $a^{\alpha}t^d \in \Fix \widetilde{\psi}^I$. For arbitrary elements $g_1, g_2 \in BS(1,n)$ if
    \[
        (g_1^{-1}g_2 g_1)\psi^{I} = g_1^{-1} a^{\alpha}t^d g_1
    \]
    holds then $u^{-1} = (g_1\varphi_x)^{-1} g' g_1$, where $g' \in BS(1,n)$ is determined by $\psi^{I}$.
\end{prop}

\begin{proof}
    Let $g_2 = a^{\gamma} t^d$, where $\gamma \in \ZZ[1/n]$, $d \in \ZZ$. We assume that 
    \begin{equation}\label{eq: WFP 1}
         (g_1^{-1} a^{\gamma} t^d g_1)\psi^{I} = g_1^{-1} a^{\alpha}t^d g_1
    \end{equation}
    holds, and so $a^{\gamma} t^d$ is a weakly fixed point of $\psi^{I}$. From Observation~\ref{obs: WFP} $a^{\gamma + \mu \beta'} t^d \sim a^{\alpha} t^d$, and let $g' = a^\delta t^r$ be the right conjugator, i.e.,
\begin{equation}\label{eq: WFP 2}
    a^\alpha t^d = (a^\delta t^r)^{-1} a^{\gamma + \mu \beta'} t^d \, (a^\delta t^r).
\end{equation}

From the construction of $\widetilde{\psi}^I$,
$$
\begin{array}{rl}
   (g_1^{-1} a^{\gamma}t^d g_1)\widetilde{\psi}^I \, = & g_1^{-1}(a\widetilde{\psi}^I, t\widetilde{\psi}^I) \, (a^{\gamma}t^d)\widetilde{\psi}^I \,  g_1(a\widetilde{\psi}^I, t\widetilde{\psi}^I) \\
    = & g_1^{-1}(uau^{-1}, uxu^{-1}) \,  u \, a^{\gamma + \mu \beta'} t^d \, u^{-1} \,  g_1(uau^{-1}, uxu^{-1})\\
    = & u \, g_1^{-1}(a,x) \, u^{-1} \, u\, a^{\gamma + \mu \beta'} t^d \, u^{-1} \, u\, g_1(a, x)\, u^{-1}\\
    = & u (g_1\varphi_x)^{-1} \, a^{\gamma + \mu \beta'} t^d \, (g_1\varphi_x) u^{-1}
\end{array}
$$
Thus using~\eqref{eq: WFP 1} and~\eqref{eq: WFP 2}  in one hand, we have
\begin{equation}\label{eq: WFP 3}
    (g_1^{-1} a^{\gamma}t^d g_1)\widetilde{\psi}^I \, = g_1^{-1} (a^\delta t^r)^{-1} a^{\gamma + \mu \beta'} t^d \, (a^\delta t^r) g_1.
\end{equation}
And, on the other hand,
\begin{equation}\label{eq: WFP 4}
    (g_1^{-1} a^{\gamma}t^d g_1)\widetilde{\psi}^I \, = u (g_1\varphi_x)^{-1} \, a^{\gamma + \mu \beta'} t^d \, (g_1\varphi_x) u^{-1}.
\end{equation}
Altogether from \eqref{eq: WFP 3} and \eqref{eq: WFP 4} we deduce the following:
%
\[
(g_1\varphi_x)\, u^{-1}\, g_1^{-1}\, (a^\delta t^r)^{-1} \, a^{\gamma + \mu \beta'} t^d = a^{\gamma + \mu \beta'} t^d \, (g_1\varphi_x)\, u^{-1}\, g_1^{-1}\, (a^\delta t^r)^{-1}
\]
Hence, $(g_1\varphi_x)\, u^{-1}\, g_1^{-1}\, (a^\delta t^r)^{-1} =1$, which in turn implies that $u^{-1} = (g_1\varphi_x)^{-1} (a^\delta t^r) g_1$. This completes the proof.
\end{proof}

The following Corollary connects the weakly outer fixed points of $\widetilde{\psi}^{I}$ and the endo-twisted conjugacy in $BS(1,n)$.

\begin{cor}
    Let $\widetilde{\psi}^I \in \End(BS(1,n))$ be defined as in~\ref{con: main con} and $a^\alpha t^d \in \Fix\widetilde{\psi}^I$. Let $w$ be the weakly outer fixed point relative to $a^\alpha t^d$ of $\widetilde{\psi}^I$. Then $u^{-1} = (g\varphi_x)^{-1} g' g$, where $g, g' \in BS(1,n)$  determined by $\widetilde{\psi}^I$ and $a^\alpha t^d$.
\end{cor}

\begin{proof}
    Since $x$ is the weakly outer fixed point relative to $a^\alpha t^d$ of $\widetilde{\psi}^I$, there exists $g \in BS(1,n)$ such that $(g^{-1} x g)\widetilde{\psi}^I = g^{-1} a^\alpha t^d g$. Let $w = a^{\gamma}t^d \in BS(1,n)$, where $\gamma \in \ZZ[1/n], \, d \in \ZZ$. We have already showed that $(a^{\gamma}t^d)\widetilde{\psi}^I = u a^{\gamma + \mu \beta'} t^d u^{-1}$. Therefore,
    $$
\begin{array}{rl}
   (g^{-1} a^{\gamma}t^d g)\widetilde{\psi}^I \, = & g^{-1}(a\widetilde{\psi}^I, t\widetilde{\psi}^I) \, (a^{\gamma}t^d)\widetilde{\psi}^I \,  g(a\widetilde{\psi}^I, t\widetilde{\psi}^I) \\
    = & g^{-1}(uau^{-1}, uxu^{-1}) \,  u \, a^{\gamma + \mu \beta'} t^d \, u^{-1} \,  g(uau^{-1}, uxu^{-1})\\
    = & u \, g^{-1}(a,x) \, u^{-1} \, u\, a^{\gamma + \mu \beta'} t^d \, u^{-1} \, u\, g(a, x)\, u^{-1}\\
    = & u (g\varphi_x)^{-1} \, a^{\gamma + \mu \beta'} t^d \, (g\varphi_x) u^{-1}
\end{array}
$$
From the hypothesis and Observation~\ref{obs: WFOP} $a^{\gamma}t^d$ is a weakly fixed point of $\widetilde{\psi}^I$ and so by Observation~\ref{obs: WFP}, $a^{\gamma+\mu\beta'}t^d \sim a^{\alpha}t^d$ with right conjugator $g'$. In other words,
\begin{equation*}
    a^{\alpha}t^d = (g')^{-1} a^{\gamma+\mu\beta'}t^d g'.
\end{equation*}
Replacing $a^{\alpha}t^d$ by above expression, in one hand we have, 
\begin{equation}\label{eq: WFOP 1}
    (g^{-1} a^{\gamma}t^d g)\widetilde{\psi}^I = g^{-1} a^\alpha t^d g = g^{-1}(g')^{-1} a^{\gamma+\mu\beta'}t^d g'g.
\end{equation}
On the other hand,
\begin{equation}\label{eq: WFOP 2}
     (g^{-1} a^{\gamma}t^d g)\widetilde{\psi}^I = u (g\varphi_x)^{-1} \, a^{\gamma + \mu \beta'} t^d \, (g\varphi_x) u^{-1}.
\end{equation}
Combining \eqref{eq: WFOP 1} and \eqref{eq: WFOP 2} we have the following:
$$
\begin{array}{ccl}
    g^{-1}(g')^{-1} a^{\gamma+\mu\beta'}t^d g'g & = &   u (g\varphi_x)^{-1} \, a^{\gamma + \mu \beta'} t^d \, (g\varphi_x) u^{-1}\\
    \Rightarrow (g\varphi_x) u^{-1}g^{-1}(g')^{-1} & = & 1
\end{array}
$$
Hence, $u^{-1} =  (g\varphi_x)^{-1}g'g$ and we get the announced result.
\end{proof}



\bigskip




\end{document}